\documentclass[12pt]{article}
\usepackage{amssymb,amsmath}
\usepackage{amsthm}

\usepackage{color}

\usepackage[latin1]{inputenc}
\usepackage[T1]{fontenc}

\usepackage{graphicx}
\DeclareGraphicsExtensions{.pdf,.png,.jpg}

\graphicspath{ {images/} }

\newcommand{\prob}{\mathbb P}
\newcommand{\eps}{\varepsilon}

\newtheorem{Theorem}{Theorem}
\newtheorem{Definition}[Theorem]{Definition}
\newtheorem{Lemma}[Theorem]{Lemma}
\newtheorem{Corollary}[Theorem]{Corollary}

\DeclareMathOperator\sgn{sgn}

\numberwithin{Theorem}{section} \numberwithin{equation}{section}

\begin{document}
\title{
Spherically Symmetric Random Permutations
}
\author {Alexander Gnedin\thanks{Queen Mary, University of London, email: a.gnedin@qmul.ac.uk}~~ and Vadim Gorin
\thanks{Massachusetts Institute of Technology and Institute for Information Transmission Problems,
email: vadicgor@gmail.com. Partially supported by NSF grant DMS-1407562 and by the
Sloan Research Fellowship.}} \maketitle

\begin{abstract}
\noindent We consider random permutations which are spherically symmetric with respect to a metric
on the symmetric group $S_n$ and are consistent as $n$ varies. The extreme infinitely spherically
symmetric permutation-valued processes are identified for the Hamming, Kendall-tau and Caley
metrics. The proofs in all three cases are based on a unified approach through stochastic
monotonicity.
\end{abstract}

\noindent
MSC:
\noindent

\noindent
\section{Introduction}
Characterisation of  processes with symmetries  as  mixtures of extreme  processes is a central theme
 in the circle of ideas surrounding de Finetti's theorem
 on infinite exchangeability.
A distinguished example is Freedman's \cite{Freedman} representation of a spherically symmetric sequence of real random variables $\xi_1,\xi_2,\dots$  as a scale mixture of i.i.d. zero-mean Gaussian sequences.
This result is equivalent to  Schoenberg's theorem from analysis,
see \cite{Berman, Kallenberg, Vitale} for background
and various proofs.

Traditionally, spherical symmetry  in $n$ dimensions   is defined as invariance of the distribution of $\xi_1,\dots,\xi_n$ under the group of orthogonal transformations.
But this property holds precisely  when  the conditional distribution on every sphere centred at the origin  is uniform.
The latter interpretation is better suitable for generalisation to metric spaces other than Euclidean and, in fact,
this kind of  extension of Freedman's theorem for $L^p$-spherically symmetric  sequences $\xi_1,\xi_2,\dots$  has appeared in the literature \cite{Berman}.
Seeking for further analogues  one is naturally lead to consider the infinite spherical symmetry in the framework of
 projective limits of metric spaces, as counterparts of the real space ${\mathbb R}^\infty$.

In this paper we   explore the setting of  combinatorial spaces of permutations
$S_n$  equipped with some  metric.
There are many meaningful metrics on $S_n$
used in applications to quantify the unsortedness of permutation  or the similarity between  rankings   \cite{Chritchlow}, but it is far from  obvious if these can be complemented by projections that
preserve the spherical symmetry.
We observe that such projections connecting the $S_n$'s exist
for  three classic metrics --  Hamming,   Kendall-tau and Cayley -- and for each of these
we identify explicitly the extreme permu\-ta\-tion-valued processes with spherical symmetry.
On the technical side, we will emphasize the approach based on stochastic monotonicity.
This method has been previously used
 in \cite{Gibbs}  in a setting closely related to ours and in
\cite{BG} in the study of Markov chains on the Young graph arising in the {asymptotic representation theory} of symmetric groups.

\section{Virtual permutations}

Suppose $S_n$  ($n=1,2,\dots$),
the symmetric group  of permutations of $[n]=\{1,\dots,n\}$,   is equipped with some metric.
Let $|\pi|$ denote the distance between $\pi\in S_n$ and the identity permutation, so
 $\{\pi\in S_n: |\pi_n|=r\}$ is a  sphere of radius $r$ centred at the identity.

A random permutation $\Pi$ of $[n]$ is just a random variable  with values in $S_n$. We say that $\Pi$ is {\it spherically symmetric }
if the probability
$\prob(\Pi=\pi)$ depends on
 $\pi\in  S_n$ only through  $|\pi|$.
For the family of spherically symmetric permutations, the random variable $|\Pi|$ is a sufficient statistic, in the sense that
 given $|\Pi|=r$ the conditional distribution of $\Pi$ is uniform on the sphere
of radius $r$.

Let $f_n:S_n\to S_{n-1}$ be a system of $n-$to$-1$ projections ($n>1$). Wherever $f_n(\pi)=\sigma$ for
 $\pi\in S_n$ and $\sigma\in S_{n-1}$ we say that  $\sigma$ is the projection of $\pi$, and  that $\pi$ is an extension of
$\sigma$. Thus every $\sigma\in S_{n-1}$ has exactly $n$  extensions in $S_n$. Generalising by
induction this relation, we define for $\nu>n$ the projection $f_{\nu\downarrow n}: S_\nu \to S_n$
through $f_{\nu\downarrow n}:=f_{n+1}\circ f_{n+2} \circ\dots \circ f_\nu$. For $\sigma\in S_n$ and
$\pi\in S_\nu$ with  $n<\nu$   we say that $\pi$ is an extension of $\sigma$ (and  $\sigma$ is the
projection of $\pi$) if $f_{\nu\downarrow n}(\pi)=\sigma$.

We shall assume throughout that the metric is {consistent} with  the projections, meaning that for all $n>1, r\geq 0$ and $\sigma\in S_{n-1}$
the number of extensions
$\#\{\pi\in S_n: |\pi|=r, f_n(\pi)=\sigma \}$ depends on $\sigma$ only through  $|\sigma|$.
The  condition is needed to ensure that  spherical symmetry is preserved by the projections, indeed the consistency entails that
\begin{equation}\label{recPi}
\prob(f_n(\Pi)=\sigma)=\sum_{r\geq 0}\sum_{\pi\in S_n: |\pi|=r,\, f_n(\pi)=\sigma}\prob(\Pi=\pi)
\end{equation}
only depends on $|\sigma|$.

The projective limit of $(S_n, f_n)$'s is the compact (in the product topology) space of
 sequences $\boldsymbol{\pi}=(\pi_1,\pi_2,\dots)$  with $\pi_n\in S_n$ and
$f_n(\pi_n)=\pi_{n-1}$ for $n>1$. We call $\boldsymbol{\pi}$ a {\it virtual
permutation}. The term is borrowed from \cite{KOV}, where a particular projective
limit of the permutation spaces  was considered. In known examples a virtual
permutation can be interpreted as a kind of combinatorial structure build upon the
infinite set ${\mathbb N}$, either a  bijection (infinite permutation) ${\mathbb
N}\to{\mathbb N}$ or a more complex object.

A random virtual permutation $\boldsymbol{\Pi}=(\Pi_1,\Pi_2,\dots)$  is a random variable, which we
canonically realize as the identity  function  on the projective limit space endowed with some
probability measure (distribution of $\boldsymbol{\Pi}$). By the measure extension theorem, the
distribution of  $\boldsymbol{\Pi}$ is uniquely determined by the marginal distributions of
$\Pi_n$'s, provided these are consistent with the projections.  We may also view $\boldsymbol{\Pi}$
dynamically as  a permutation-valued growth process, where  $\Pi_n$ extends $\Pi_{n-1}$ in some
random fashion.

Special notation  $\boldsymbol{\Pi^*}=(\Pi_1^*, \Pi_2^*,\dots)$ will be used
for the {\it uniform} virtual permutation which has $\Pi_n^*$ uniformly distributed over $S_n$ for every $n$; the consistency in this case follows from the
$n-$to$-1$ property of $f_n$'s.   To construct $\boldsymbol{\Pi}^*$ sequentially, at each step  the extension  must be chosen   from the available options uniformly.
It should be stressed that the support and distribution of $\boldsymbol{\Pi}^*$ depend on  the type of  projections.

In the sequel we focus on  {\it infinitely spherically symmetric} (ISS) random
virtual permutations $\boldsymbol{\Pi}$, which have every $\Pi_n, n=1,2,\dots,$  spherically symmetric. Clearly,   the uniform $\boldsymbol{\Pi}^*$ is ISS.

The sequence of sufficient statistics $|\boldsymbol{\Pi}|:=(|\Pi_1|,|\Pi_2|,\dots)$ is  a Markov
chain  naturally associated with ISS virtual permutation. The Markov property  for
$|\boldsymbol{\Pi}|$ is readily justified by looking at the time-reversed process. Moreover, the
backward transition probabilities for $|\boldsymbol{\Pi}|$ do not depend on particular
$\boldsymbol{\Pi}$, hence  are the same as for  the uniform $\boldsymbol{\Pi}^*$. Conversely, if
a time-inhomogeneous Markov chain at every step has the same range and the same backward transition
probabilities as $|\boldsymbol{\Pi}^*|$, then it can be uniquely realised as $|\boldsymbol{\Pi}|$
for some ISS virtual permutation.

The set of spherically symmetric distributions for each $\Pi_n$ is a simplex.
The family of ISS virtual permutations is  a projective limit of these finite-dimensional simplices,
thus by a general  result
(see e.g. \cite{Goodearl}, p. 164)
  it is a Choquet simplex, i.e. a convex
compact set with the property that each element has a unique representation as a mixture of extreme
elements. In this sense the problem of describing the ISS virtual permutations amounts to
identifying the extremes.

There is one very general approach to the problem, which can be traced back deeply in history. By a
theorem attributed to Maxwell and Borel  \cite{Kallenberg}, the projection to $n$ dimensions of the
uniform distribution on a sphere in ${\mathbb R}^{\nu}$ of radius $\lambda{\nu}^{1/2}$ converges,
as $\nu\to\infty$,  to the product of $n$ copies of  ${\cal N}(0,\lambda^2)$. Comparing with
Freedman's theorem, it follows that
 all extreme ISS sequences in the Euclidean setting appear  as limits of  such projections  from spheres in high dimensions.
Likewise, let
${\mathfrak{U}}_{\nu, r}$ be a uniformly random element of the sphere  of radius $r$ in the symmetric group $S_\nu$.
Restating another general result (cf. \cite{DF}, Theorem 4.1)   we have the following analogue.

\begin{Theorem} \label{Theorem_Martin_approximation}
 If $\boldsymbol{\Pi}=(\Pi_1,\Pi_2,\dots)$  is an extreme ISS virtual permutation, then there exists a sequence of numbers $r(\nu)$, $\nu=1,2,\dots$ such that for each $n=1,2,\dots$
the sequence  $ f_{\nu\downarrow n}( {\mathfrak{U}}_{\nu, r(\nu)})$ converges in distribution to $\Pi_n$
as $\nu\to\infty$.

\end{Theorem}
\noindent
The set of probability distributions
that arise as such limits is called, depending on the context, the {Martin
boundary} or the family of Boltzmann laws.

Here is another simple property of the Euclidean spheres, which can be used to
give yet another proof of Friedman's theorem.
Let $U$ be uniformly distributed on the unit sphere in ${\mathbb R}^{\nu}$, and let
$\rho_n$ be the norm of the coordinate projection of $U$ to ${\mathbb R}^n$ for
$n<\nu$. Then $rU$ is uniform on the sphere of radius $r$,  and $r\rho_n$ is the norm of the projection of $rU$ in $n$ dimensions. Note that
 the random variable $r\rho_n$  increases with $r$.
For the symmetric group,
an analogous property in the form of {\it stochastic monotonicity} of $|f_{\nu\downarrow n}( {\mathfrak{U}}_{\nu, r})|$ in $r$
 holds for all three metrics we consider here
 (cf. Lemmas \ref{Lemma_Hamming_mono}, \ref{Lemma_Mallows_mono}, \ref{Lemma_Ewens_mono}).
We use  the property as a key tool to identify  the extreme ISS virtual permutations.

\section{Hamming spherical symmetry}

\label{Section_Hamming}

\subsection{Classification theorem}

Hamming distance between two permutations $\pi$ and $\sigma$
in $S_n$ is defined as the number of positions $j\in[n]$ where $\pi(j)\neq\sigma(j)$.
 The distance to the identity permutation is $|\pi|=n-F(\pi)$, where
$F(\pi):  =\#\{j\in [n]: \pi(j)\neq j\}$ is the number of fixed points in $\pi$. Therefore, a
sphere in the Hamming distance can have radius $0,2,3,\dots,n$. The sphere of radius $0$ has a
single element, which is  the identity permutation. The elements of the sphere of radius $n$ have no fixed
points, they are called \emph{derangements.} Thus a random permutation $\Pi$ is Hamming-spherically
symmetric if   its distribution is conditionally uniform given the number of fixed points.

Let $f_n:S_n\to S_{n-1}$ be the operation of deleting  element $n$ from permutation written in the
cycle notation. Then   permutation $\sigma\in S_{n-1}$ has $n$ extensions obtained by either
inserting element $n$ in a cycle  next to the right of one of the elements in $\sigma$, or
appending a singleton cycle $(n)$. For instance, five extensions of $(13)(24)\in S_4$ are
$$(153)(24), (135)(24), (13)(254), (13) (245), (13)(24)(5).$$
If a permutation $\sigma\in S_{n-1}$ has $k$ fixed points, then it has $1$ extension
with $k+1$ fixed points, $k$ extensions with $k-1$ fixed points and $n-k-1$
extensions with $k$ fixed points. Hence, the Hamming spherical symmetry is
consistent under this system of projections.

The virtual permutations defined via the $f_n$'s  were introduced in \cite{KOV}.
Writing permutation in the cycle notation yields  a composite combinatorial
structure, comprised of a partition of the set $[n]$ into disjoint nonempty blocks
and a linear order  on each block of the partition, with the property that in each
block the smallest integer is also the minimal element of the order. Similarly,  a
virtual permutation corresponds to a partition of $\mathbb N$ into some collection
of disjoint nonempty blocks, taken together with  a linear order on each block, such
that within the block the smallest integer  is also the minimal element in the
linear order.

The sequential construction of uniform virtual permutation $\boldsymbol{\Pi}^*$
specializes  in the cycle notation as the Dubins-Pitman  {\it Chinese restaurant
process} \cite{CSP} with the following dynamics: given the permutation at step $n-1$
is $\Pi_{n-1}=\sigma$, element $n$ is inserted with probability $1/n$ in a cycle
next to the right of any given element of $\sigma$, and appended to $\sigma$ as
singleton cycle $(n)$ with probability  $1/n$. This process has been intensely
studied, in particular, it is well known that the random series comprised of the
asymptotic frequencies of blocks follows the Poisson-Dirichlet/GEM distribution with
parameter 1 \cite{ABT, CSP}. To this we only add here that $\boldsymbol{\Pi}^*$ can
be seen as a  partition of $\mathbb N$ in infinitely many blocks, with the set of
elements within each block  ordered like the set of nonnegative rational numbers.

\smallskip
We introduce next a family of ISS virtual permutations
$\{\boldsymbol{\Pi}^{\alpha}, \alpha\in[0,1]\}$, which includes
the uniform virtual permutation as the $\alpha=0$ case.
Another edge case, $\alpha=1$, corresponds to the  trivial virtual
permutation, which restricts to every $[n]$ as the identity.

For $\alpha\in(0,1)$, to construct the
virtual permutation $\boldsymbol{\Pi}^{\alpha}$ explicitly, it will be convenient  to introduce
enriched permutations of $[n]$, which have an additional feature
 that genuine   singleton cycles are distinguished from the cycles
which will be bigger  within a larger context $[\nu]\supset[n]$ but have a sole representative in
$[n]$. For $\boldsymbol{\pi}$ a virtual permutation, call element $n$ {\it singular} if $(n)$ is a
singleton cycle in every $\pi_\nu, ~\nu\geq n$, and call $n$ {\it regular} otherwise. Define
enriched virtual permutation $\tilde{\boldsymbol{\pi}}=(\tilde{\pi}_1, \tilde{\pi}_2,\dots)$ to be
a virtual permutation $\boldsymbol{\pi}=(\pi_1,\pi_2,\dots)$ with additional classification of the
elements in each $\pi_n$ in singular and regular, and such that
 $\tilde{\pi}_n$ is consistent under the deletion of elements from cycles.
For instance $\tilde{\pi}_6=(13)({\it 2})(46)(5)$ with singular element ${ 2}$  and regular
other five elements encodes that some elements $\nu>6$ will be inserted in the cycle of every
element except $2$.
The correspondence between
$\boldsymbol{\pi}$ and $\tilde{\boldsymbol{\pi}}$ is a canonical bijection, but for any fixed $n$,
$\tilde{\pi}_n$ as compared to $\pi_n$ contains more information about $\boldsymbol{\pi}$.

For $\alpha\in[0,1)$ define a random enriched virtual permutation by the following two rules.
\begin{itemize}
\item[(i)] Each $n$ independently of other elements is singular with probability $\alpha$.
\item[(ii)]  The virtual permutation restricted to the set of regular elements is  distributed like the uniform  $\boldsymbol{\Pi}^*$, provided the regular elements are enumerated
in increasing order by $\mathbb N$.
\end{itemize}
Here is a sequential construction of $\tilde{\boldsymbol\Pi}^\alpha$, modifying the Chinese restaurant process. Element 1 is singular
with probability $\alpha$. Inductively, for $n>1,$  as an enriched permutation of $[n-1]$ with some
$s$ singular fixed points  (hence $n-s-1$ regular elements) has been constructed, element $n$
becomes singular with probability $\alpha$, is inserted in existing cycle next to the right of any
given regular element with probability $(1-\alpha)/(n-s)$, and is
 appended  as a regular singleton cycle with probability $(1-\alpha)/(n-s)$.

Discarding the division into regular and singular elements
yields $\boldsymbol{\Pi}^\alpha$. Explicitly,
for the probability
$p_{n,k}:=\prob (\Pi^\alpha_n=\sigma)$
 of a permutation $\sigma\in S_n$ with $k$ fixed points we have
\begin{equation}\label{eq_Hamming_coherent}
 p_{n,k}{(\alpha)}=\sum_{j=0}^k {k\choose j}
\alpha^j(1-\alpha)^{n-j} \,\frac{1}{(n-j)!}\,,~~ ~k\in\{0,1,\ldots,n-2,n\}.
\end{equation}
The formula follows by noting that the probability of any given enriched permutation with $j$
singular elements is $\alpha^j(1-\alpha)^{n-j}/(n-j)!$, and that any $j$ out of $k$ fixed points
can be singular. It is obvious from  \eqref{eq_Hamming_coherent}  that $\boldsymbol
\Pi^\alpha$ is ISS.

\begin{Lemma} \label{Lemma_Hamming_regularity}
The virtual permutation $\boldsymbol{\Pi}^\alpha$ satisfies
$$\lim_{n\to\infty} \frac{F(\Pi_n^\alpha)}{n}\to\alpha$$
almost surely.
\end{Lemma}
\begin{proof} For any virtual permutation
$$\prob(F(\pi_{n-1})=k-1|F(\Pi_n)=k)=\frac{k}{n}\,, ~~\prob(F(\pi_{n-1})\geq k|F(\Pi_n)=k)=\frac{n-k}{n},$$
which readily implies that the sequence $F(\Pi_n)/n$ is a reverse submartingale, hence converges almost surely. For $\boldsymbol{\Pi}^*$ we have ${\mathbb E}[F(\Pi^*_n)]=1$ hence the limit fraction is $0$.
The limit for $\alpha\neq 0$ follows from this and the construction of  $\boldsymbol{\Pi^\alpha}$. \end{proof}

\begin{Theorem} \label{Theorem_Hamming}
The extreme Hamming-ISS virtual permutations are $\{\boldsymbol \Pi^{\alpha}, \, \alpha\in[0,1]\}$.
\end{Theorem}

In the sequel
we give two    proofs of Theorem \ref{Theorem_Hamming} using different techniques. The first  proof
in Section
\ref{Section_singletons_count}
is based on the explicit enumeration of spheres.
The second proof in Section \ref{Section_singletons_mono}
uses stochastic monotonicity.

\subsection{Proof through exact enumeration} \label{Section_singletons_count}

A starting point for a straight approach to extreme virtual permutations is the enumeration of Hamming
spheres. The sphere of radius $n$ is the set of {\it derangements}, which are permutations with no
fixed points. The number of derangements $d_n$ is given by
\begin{equation}\label{deM}
d_n= n!\sum_{j=0}^n \frac{(-1)^j}{j!}=\left\lfloor \frac{n!}{e}+\frac{1}{n}\right\rfloor.
\end{equation}
The first part of this formula is the classics due to de Montmort (1713), and the second  is found in \cite{Hassani}.
Denoting $D_{n,k}$  the number of permutations of $[n]$ with $k$ fixed points, we have
\begin{equation}\label{Dd}
D_{n,k}={n\choose k} d_{n-k}.
\end{equation}

Note that the elements comprising singleton cycles are precisely the fixed points of permutation.
If $F(\sigma)=k$ for $\sigma\in S_{n-1}$ then  $\sigma$ has $n-k-1$ extensions $\pi\in S_n$ with $F(\pi)=k$,  one extension with $F(\pi)=k+1$ and $k$ extensions with  $F(\pi)=k-1$.
Reciprocally,  enumerating projections for given  $\pi\in S_n$ with $F(\pi)=k$
yields
the recursion
\begin{equation}\label{D-Ham}
D_{n,k}=(n-k-1)D_{n-1,k}+D_{n-1,k-1}+(k+1)D_{n-1,k+1}, ~~~0\leq k\leq n,
\end{equation}
where $D_{1,0}=0, D_{1,1}=1$ and we adopt the convention $D_{n,j}=0$ for $j\in \{-1, n+1\}$. The
recursion implies  $D_{n, n-1}=0$, in accord with the fact that  permutation of $[n]$ cannot have
$n-1$ fixed points. There is some similarity  between (\ref{D-Ham}) and  a two-term recursion for
the Eulerian numbers counting descents \cite{GOEuler}, but these have very
different properties.

\bigskip

For $\boldsymbol{\Pi}=(\Pi_1,\Pi_2,\dots)$ a ISS virtual permutation, let $p_{n,k}=\prob(\Pi_n=\pi)$ be the probability of any given permutation $\pi\in S_n$ with $F(\pi)=n-|\pi|=k$ fixed points.
We call the bivariate array  ${\boldsymbol p}=(p_{n,k})$ the {\it probability function}.
By the rule of addition of probabilities
\begin{equation}\label{p-Ham}
p_{n,k}=(n-k)p_{n+1,k}+p_{n+1,k+1}+k p_{n+1,k-1},~~~k\in\{0,1,\ldots,n-2,n\},
\end{equation}
which is a backward recursion, dual to (\ref{D-Ham}). Every nonnegative solution to (\ref{p-Ham}) with $p_{1,1}=1$
is a probability function for some unique
 random virtual permutation.
Observe that
for $n$ fixed,
$D_{n,k}p_{n,k}, 0\leq k\leq n$, is the distribution of  the number of fixed points $F(\Pi_n)=n-|\Pi_n|$,  in particular
$$\sum_{k=0}^n D_{n,k}p_{n,k}=1,$$
and $D_{n,0}p_{n,0}=d_n p_{n,0}$ is the probability that $\Pi_n$ is a derangement.

Now we wish to explore the limit distributions which can arise in Theorem
\ref{Theorem_Martin_approximation}. To that end, for integer $\nu$ and $0\leq \varkappa<\nu-1$ let
\begin{equation}\label{incomplete}
p^{\nu,\varkappa}_{n ,  k }=\prob(\Pi_n^*=\pi\,|\, F(\Pi^*_\nu)=\varkappa), {\rm~~~where~~}    \pi\in S_n, F(\pi)=k.
\end{equation}
In terms of the Markov chain $(F({\Pi}^*_n), n=1,2,\dots)$
$$
D_{n,k}p^{\nu,\varkappa}_{n ,  k } =\prob(F(\Pi_n^*)=k\,|\, F(\Pi^*_\nu)=\varkappa)
$$
is the backward transition probability (also for any other ISS  virtual permutation in place of
$\boldsymbol{\Pi}^*$). Viewed as a function of $n$ and $k$, $p^{\nu,\varkappa}_{n,k}$  is an
incomplete probability function  which satisfies (\ref{p-Ham}) for $n<\nu$ together with  the
boundary conditions $p^{\nu,\varkappa}_{\nu ,  \varkappa } =1/D_{\nu,\varkappa}$ and
$p^{\nu,\varkappa}_{\nu ,  k } =0$ for $k\neq\varkappa$.

Let $D_{n,k}^{\nu,\varkappa}$ be the number of extensions of a  given $\sigma\in S_n$ with
$F(\sigma)=k$ to any $\pi\in S_\nu$ with $F(\pi)=\varkappa$. Then
\begin{equation}\label{rat-Ham}
p^{\nu,\varkappa}_{n ,  k }= \frac{D_{n,k}^{\nu,\varkappa}}{D_{\nu,\varkappa}}\,, ~~~1\leq n\leq \nu,
\end{equation}
where the ratio is called the  {\it Martin kernel}. To find  probability functions
appearing as limits of the Martin kernel one needs to identify the regimes for
$\varkappa=\varkappa(\nu)$ which ensure convergence of (\ref{rat-Ham}) as
$\nu\to\infty$, for all fixed $n$ and $k$. In the case of convergence, we will say
that the limit probability function is {\it induced} by $\varkappa(\nu)$ as
$\nu\to\infty$. The following simple observation allows us to only focus  on the
probabilities of derangements.

\begin{Lemma}\label{L1} Every solution to {\rm (\ref{p-Ham})} is uniquely determined by
$(p_{n,0},~ n=1,2,\dots)$.
\end{Lemma}
\begin{proof}  Re-write the recursion as
 $p_{n+1,k+1}=p_{n,k}-(n-k)p_{n+1,k}-kp_{n+1,k-1}$, and the conclusion is obvious by induction in $k$.
\end{proof}

Now,
 $D_{n,0}^{\nu,\varkappa}$ counts  extensions of  a given derangement $\sigma\in S_n$  to a permutation $\pi\in S_\nu$ with $\varkappa$ fixed points.
Clearly,  $D_{n,0}^{\nu,\varkappa}=0$ if $\varkappa>\nu-n$. Otherwise, out of $\nu-n$ elements
 added to $\sigma$ some $m\leq \nu-n-\varkappa$ elements, say $a_1<\dots<a_m$
are allocated within the cycles present
in $\sigma$, and there are $n(n+1)\cdots (n+m-1)$ such allocations.
The other $\nu-n-m$ elements $\{n+1,\dots,\nu\}\setminus\{a_1,\dots,a_m\}$
form new cycles of which $\varkappa$ are singletons.
Thus using (\ref{Dd}) we obtain

\begin{eqnarray*}
 D_{n,0}^{\nu,\varkappa}=\sum_{m=0}^{\nu-n-\varkappa} {\nu-n\choose m} \frac{(n+m-1)!}{(n-1)!}\,D_{\nu-n-m,\varkappa}=\\
\sum_{m=0}^{\nu-n-\varkappa} {\nu-n\choose m} \frac{(n+m-1)!}{(n-1)!} {\nu-n-m\choose \varkappa}   \,d_{\nu-n-m- \varkappa}=\\
\frac{(\nu-n)!}{\varkappa!} \sum_{m=0}^{\nu-n-\varkappa} {n+m-1\choose m}\frac{d_{\nu-n-m- \varkappa}}{(\nu-n-m- \varkappa)!}\,.
\end{eqnarray*}

\begin{Lemma}\label{L2} As $\nu\to\infty$, the ratios {\rm(\ref{rat-Ham})} converge  for  $k=0$ (hence also for $k\geq 0$)   if and only if $\varkappa/\nu\to\alpha$ for some $\alpha\in[0,1]$.
In this case

\begin{equation}\label{pbeta0}
p^{\nu,\varkappa}_{n ,0} \to  \frac{(1-\alpha)^n}{n!}\,.
\end{equation}

\end{Lemma}

\begin{proof} Applying (\ref{deM}) and (\ref{Dd}) and using the above calculation
for  $D_{n,0}^{\nu,\varkappa}$, after cancelations we get as $\nu\to\infty$

$$
p^{\nu,\varkappa}_{n ,0}= \frac{D_{n,0}^{\nu,\varkappa}}
{D_{\nu,\varkappa}}\sim
\frac{1}{\nu^n} \sum_{m=0}^{\nu-n-\varkappa} {n+m-1\choose m}\sim
 \int_0^{1-\varkappa/\nu} \frac{x^{n-1}}{(n-1)!}\,dx=\frac{(1-\varkappa/\nu)^n}{n!},
$$
thus the convergence   holds if  and only if $\varkappa/\nu\to \alpha$ for some $\alpha\in[0,1]$.
\end{proof}

We see that there exists a unique probability function  ${\boldsymbol
p}(\alpha)=(p_{n,k}(\alpha))$,  with  $p_{n,0}(\alpha)=(1-\alpha)^n/n!$. Comparing with
\eqref{eq_Hamming_coherent} we conclude that ${\boldsymbol p}(\alpha)$ is the probability function
of the virtual permutation $\boldsymbol \Pi^{\alpha}$.

\begin{proof}[Proof of Theorem \ref{Theorem_Hamming}] Since ${\boldsymbol p}(\alpha)$'s (corresponding to $\boldsymbol \Pi^{\alpha}$) are the only possible limits in the context of Theorem \ref{Theorem_Martin_approximation}, the family contains all extremes.
On the other hand, by Lemma \ref{Lemma_Hamming_regularity}, $F(\boldsymbol
\Pi_n^\alpha)/n\to \alpha$ a.s. Thus the supports of distributions corresponding to different $\alpha$ are
disjoint, and each ${\boldsymbol p}(\alpha)$ is extreme.
\end{proof}


\subsection{Proof through monotonicity} \label{Section_singletons_mono}

For the second proof we recall the notion of stochastic order.

\begin{Definition} For real random variables $\xi$ and $\eta$ we say that
  $\xi$ is stochastically larger than $\eta$, denoted   $\xi\geq_{\rm st}\eta$ if either of the two
 equivalent properties hold:
 \begin{itemize} \item[\rm(a)] For each $x\in\mathbb R$, we have $\prob (\xi \ge x)\ge \prob (\eta\ge x)$,
\item[\rm(b)] ${\mathbb E}[u(\xi)]\geq {\mathbb E}[u(\eta)]$ for every nondecreasing function $u$.
 \end{itemize}
\end{Definition}
\noindent
If $\xi\geq_{\rm st}\eta$, then it is possible to define distributional copies of
these variables, say $\xi'$ and $\eta'$, on the same probability space in such a way
that $\xi'\ge \eta'$ almost surely.

\begin{Lemma} \label{Lemma_Hamming_mono}
 Let $\Pi$ and $\Pi'$ be two random Hamming--spherically symmetric permutations in $S_\nu$, such that
$F(\Pi)\geq_{\rm st} F(\Pi').$ Then also $F(f_{\nu\downarrow n}(\Pi))\geq_{\rm st}  F(f_{\nu\downarrow n}(\Pi'))$  for $n<\nu$.
\end{Lemma}
\begin{proof}
It is sufficient to prove the relation for $n=\nu-1$, with  $F(\Pi)$ and $F(\Pi')$  some given nonrandom values.
Excluding the trivial case  $F(\Pi)=\nu$ of the identity permutation,  we further reduce to the case  $F(\Pi')=k, F(\Pi)=k+1$ with $0\leq k<\nu-2$.
The general case will follow by induction and using the fact that the stochastic order is preserved by convex mixtures.

We have   $F(f_\nu(\Pi))\in \{k,k+1,k+2\}$ and   $F(f_\nu(\Pi'))\in \{k-1,k,k+1\}$, thus  to show that
$F(f_\nu(\Pi))$  is stochastically larger than $F(f_\nu(\Pi'))$
we only need to check that
 \begin{equation}
 \label{eq_x1}
 \prob (F(f_\nu(\Pi'))=k+1)\le \prob (F(f_\nu(\Pi))\in
 \{k+1,k+2\}).
 \end{equation}
Since $\Pi$ is uniformly distributed over permutations with $k+1$ fixed points,  and
since the event in the right-hand side  of (\ref{eq_x1}) occurs precisely when $\nu$
is not a  fixed point, the  probability of this event is $(\nu-k-1)/\nu$. Likewise,
the event on the left-hand side implies that  $\nu$ is neither a fixed point nor
belongs  to a $(\nu-k)$-cycle of $\Pi'$. The probability that $\nu$ is a fixed point
of $\Pi'$ is $k/\nu$. The probability that $\Pi'$ has a $(\nu-k)$-cycle containing
$\nu$ is
$$\frac{(\nu-k)}{\nu}\, \frac{(\nu-k-1)!}{d_{\nu-k}}<\frac{(\nu-k)}{\nu} \frac{(\nu-k-1)!}{(\nu-k)!}=\frac{1}{\nu},$$
where we used the obvious bound $d_{\nu-k}<(\nu-k)!$ along with the fact that given
the cycle structure $\nu$ is equally likely to occupy any position within the
cycles. Now (\ref{eq_x1}) follows:
$$
\prob (F(f_\nu(\Pi'))=k+1)\le 1-\frac{k}{\nu}-\frac{1}{\nu}=
\frac{\nu-k-1}{\nu}=\prob (F(f_\nu(\Pi))\in
 \{k+1,k+2\}).$$\end{proof}


\begin{proof}[Proof of Theorem \ref{Theorem_Hamming}]

Let $\boldsymbol{\Pi}=(\Pi_n, n=1,2,\dots)$ be extreme Hamming-ISS virtual
permutation. There exists a sequence $\varkappa(\nu)$ such that $\boldsymbol{\Pi}$
is representable as a limit of ${\mathfrak{U}}_{\nu, \nu-\varkappa(\nu)}$ as in
Theorem \ref{Theorem_Martin_approximation} (where $r(\nu)=\nu-\varkappa(\nu)$).
Passsing if necessary to a subsequence we may assume that $\varkappa(\nu)/\nu\to
\alpha$ as $\nu\to\infty$ for some $\alpha\in [0,1]$.

Suppose first $\alpha\in (0,1)$ and  choose $0<\epsilon<\min(\alpha,1-\alpha)$.
As $\nu\to\infty$, $F(f_{\nu\downarrow n}\, {\mathfrak{U}}_{\nu, \nu-k(\nu)})$
converges in distribution to $F(\Pi_n)$, and by Lemma \ref{Lemma_Hamming_regularity}
the probability of relation $F(\Pi_\nu^{\alpha+\epsilon})>\varkappa(\nu)$ approaches
1. Likewise, the probability of relation
$F(\Pi_\nu^{\alpha-\epsilon})<\varkappa(\nu)$ approaches 1. Invoking  Lemma
\ref{Lemma_Hamming_mono} we obtain for projections
$$F(\Pi_n^{\alpha-\epsilon})\leq_{\rm st} F(\Pi_n) \leq_{\rm st} F(\Pi_n^{\alpha+\epsilon}).$$
By continuity in the parameter both bounds  converge in distribution to $F(\Pi_n^\alpha)$ as $\epsilon\to 0$. Thus $F(\Pi_n)$ has the same distribution
as $F(\Pi_n^\alpha)$,  implying that  $\Pi_n$ and $\Pi_n^\alpha$ have the same distribution for every $n$. It follows that $\boldsymbol{\Pi}$ has the same distribution as $\boldsymbol{\Pi}^\alpha$.

The edge cases $\alpha\in\{0,1\}$ are treated similarly, with one-sided bounds derived from ${\Pi}_\nu^{\epsilon}$ and ${\Pi}_\nu^{1-\epsilon}$, respectively.

It follows that  virtual permutations $\{\boldsymbol{\Pi}^\alpha, ~\alpha\in
[0,1]\}$ are the only possible limits in Theorem \ref{Theorem_Martin_approximation}.
Since their supports are disjoint by Lemma \ref{Lemma_Hamming_regularity}, this is
the set of extremes.
\end{proof}

\subsection{Complements}

\paragraph{1.} For $\tilde{\boldsymbol{\Pi}}^\alpha$
the   bivariate process counting  singular and regular fixed points is a  Markov chain with transition probabilities
at step $n$ being

\begin{eqnarray*}
(s, r)
\to\begin{cases}
(s+1,r) {\rm  ~w.p.~} \alpha\\
(s,r+1) {\rm  ~w.p.~} \frac{1-\alpha}{n-s}\\
(s,r-1){\rm  ~w.p.~} \frac{(1-\alpha)r}{n-s}\\
(s, r){\rm  ~~~~~~w.p.~}  \frac{(1-\alpha)(n-1-r-s)}{n-s}.
\end{cases}
\end{eqnarray*}
For the count of fixed points $F(\Pi_n^\alpha)$
the transition probabilities are more involved.

\paragraph{2.} With ISS virtual permutation $\boldsymbol{\Pi}$ one can uniquely associate a partition of the infinite set $\mathbb N$, by assigning integers $i$ and $j$ to the same block if they  belong to the same cycle of $\Pi_n$ for $n\geq\max(i,j)$.
This partition is exchangeable, that is has distribution invariant under bijections ${\mathbb N}\to{\mathbb N}$ moving finitely many elements.
Partitions with nonzero frequency of singletons (dust component)  appear as intermediate states in exchangeable coalescence processes (e.g. \cite{GIM}).

The distribution of exchangeable partition is determined by  the exchangeable partition probability function (EPPF), see  \cite{CSP}.
Our classification of ISS virtual permutations can be recast
as the characterisation  of partitions of $\mathbb N$ with EPPF
of the form
$$p(n_1,\dots,n_\ell)=p_{n,k} \prod_{j=1}^\ell (n_j-1)!,$$
where $n_1,\dots,n_\ell$ is a partition of integer $n$ and $k$ is the multiplicity of part 1.

\paragraph{3.}  It is well known that
for the uniform $\Pi_n^*$, the sequence of cycle sizes arranged in non-increasing order and
normalised by $n$ converges weakly to the Poisson-Dirichlet  distribution with parameter 1
\cite{ABT, CSP}. The same limit holds for $\Pi_n^\alpha$, provided the
cycle sizes are normalised by $(1-\alpha)n$.

\section{Kendall-tau spherical symmetry}

\subsection{Inversions in permutation}
 The Kendall-tau distance  between $\pi$ and $\sigma$ in $S_n$
is the number of discordant pairs,  i.e. positions $i<j$ with
$\sgn(\pi(i)-\pi(j))=-\sgn(\sigma(i)-\sigma(j))$.
When $\sigma$ is the identity permutation,   discordant pair $i<j$  is   inversion in $\pi$, thus $|\pi|$ coincides with the
number of inversions
$$I(\pi):=\#\{(i,j): 1\leq i<j\leq n: \pi(i)>\pi(j)\}.$$
In this setting  spherical symmetry means that permutations of $[n]$ with the same number of inversions have equal probability.

Viewing permutation $\pi\in S_n$ as a linear order on the set of {\it positions} and using the one-line notation $(\pi(1),\dots,\pi(n))$,  we  understand $\pi(j)$ as the rank
of element $j$ among $[n]$ (so $\pi(j)=1$ if $j$ is the minimal element in the order). It is useful to observe that for the inverse permutation $I(\pi^{-1})=I(\pi)$, which suggests two systems of projections, each consistent
with the spherical symmetry:

\begin{itemize}
\item[(i)] $f'_n(\pi)$ deletes the last entry $\pi(n)$ , and re-labels  $\pi(1),\dots,\pi(n-1)$ by an increasing bijection with $[n-1]$.
\item[(ii)] $f''_n(\pi)$ deletes letter $n$ from the one-line notation.
\end{itemize}
 For instance,
$$f_5': (2,5,1,4,3)\mapsto (2,4,1,3), ~~~~f_5'': (2,5,1,4,3)\mapsto (2,1,4,3).$$
The projections are mapped into one another by the group inversion.

For convenience we will work with  projection $f'_n$, which may be also seen as the restriction of
order from $[n]$ to $[n-1]$. The advantage of this choice of projection is that the set of
inversions within $[n]$ remains unaltered as the permutation gets extended. Furthermore, the
mapping $(\pi(1),\dots,\pi(n))\mapsto (n-\pi(1),\dots,n-\pi(n))$ yields the inverse order relation,
hence also consistent with the projections.

A virtual permutation $\boldsymbol{\pi}=(\pi_1,\pi_2,\dots)$ in this setting  is a system of consistent orders on $[n]$  for $n=1,2,\dots$, hence  defines a linear order on $\mathbb N$.
Note that $\pi_{n+1}(j)-\pi_n(j)$ is $1$ or $0$ depending on whether $\pi_{n+1}(n+1)\leq j$ or not.
The order defined by $\boldsymbol{\pi}$ is a well-order (i.e. isomorphic to $({\mathbb N},\leq$)) if and only if $\pi_n(j)$ has a terminal value as $n$ increases, for every $j$.

Observe that
 $I(\pi)=\sum_{j=1}^n \eta_j$, where
 $\eta_j:=\#\{i: 1\leq i<j,~ \pi(i)>\pi(j)\}.$
The mapping  $\pi\mapsto\eta_1,\dots,\eta_n$   is a bijection called the {\it   Lehmer code}.
For instance, $(2,5,1,4,3)$ is encoded into  $0,0,2,1,2$.
In terms of the Lehmer code, $f_n'$  acts as  the coordinate projection which sends $\eta_1,\dots,\eta_{n-1}, \eta_{n}$ to $\eta_1,\dots,\eta_{n-1}$.
The consistency with the  Kendall-tau distance is now easily seen.
Indeed,
if $\pi\in S_n$  is an extension of $\sigma\in S_{n-1}$, the counts of inversions are related as $I(\pi)=I(\sigma)+n-\pi(n)=  I(\sigma)+\eta_n$.
 The idea of the Lehmer code generalises to the infinite setting, allowing us to encode the  order on $\mathbb N$ in a single string $\eta_1,\eta_2,\dots$.

\subsection{Mallows distributions}

For the virtual permutation $\boldsymbol{\Pi}^*$ the terms of the Lehmer code $\eta_1,\eta_2,\dots$
are independent, with  $\eta_j$ being uniformly distributed on $\{0,1, \dots, j-1\}$. This connects
the problem of classification of the ISS permutations to the study of conditional laws and tail
algebras for sums of independent random variables, see \cite{Aldous, Mineka} and especially
\cite{BH}.

Exponential tilting of  the uniform distribution yields a truncated geometric distribution with masses $q^{i}/[j]_q~$ $(0\leq i\leq j-1)$, where $[j]_q:=(1-q^j)/(1-q)$.
The tilted joint distribution of $\eta_1,\dots,\eta_n$ conditional on $\eta_1+\dots+\eta_n$ does not depend on $q$, hence the corresponding permutation is Kendall-tau spherically symmetric,
with distribution
\begin{equation}\label{Mallows}
{\mathbb P}(\Pi_n=\pi)=  \frac{q^{I(\pi)}}{[n]_q!\,}, ~~\pi\in S_n,
\end{equation}
where $[n]_q!:=\prod_{j=1}^n [j]_q$.  This is the {\it Mallows distribution} on
permutations. The parameter range is $q\in [0,\infty]$, where the edge cases $0$ and
$\infty$ correspond to the deterministic identity and the decreasing permutation
$(n, n-1,\dots,1)$, respectively.

Under the Mallows distribution with $q<1$ the virtual permutation determines a well-order
on $\mathbb N$. Passing to the
inverse order relation yields a Mallows distribution with parameter $1/q$. See \cite{BP, GO}
for further properties of the Mallows permutations.

\begin{Theorem} \label{Theorem_Mallows}
Mallows distributions {\rm (\ref{Mallows})} with $q\in[0,\infty]$ and only they are the extreme ISS
virtual  permutations with respect to the Kendall-tau distance.
\end{Theorem}



\noindent
The proof  hinges on the following analogue of Lemma \ref{Lemma_Hamming_mono}.

\begin{Lemma} \label{Lemma_Mallows_mono}
 Let $\Pi$ and $\Pi'$ be two random Kendall--tau--spherically symmetric permutations in $S_\nu$, such that
$I(\Pi)\geq_{\rm st} I(\Pi').$ Then also $I(f'_{\nu\downarrow n}(\Pi))\geq_{\rm st}  I(f'_{\nu\downarrow n}(\Pi'))$  for $n<\nu$.
\end{Lemma}

\noindent
To show this we will need the following property of the uniform distribution.

\begin{Lemma} \label{Lemma_stochastic_nonsense} For $i=1,2,3$ let $(U_i,V_i)$ be pairs of integer random variables such that
\begin{itemize}
\item[\rm (i)]    $V_1$ and $U_1$  are independent, with $U_1$ uniformly  distributed on some integer interval,

\item[\rm(ii)] the conditional distribution of $(V_i,U_i)$ given $V_i+U_i$ is the same for $i=1,2,3$,
\item[\rm(iii)]  $V_2+U_2\geq_{\rm st} V_3+U_3$.
\end{itemize}
Then $V_2\geq_{\rm st}V_3$.
\end{Lemma}

\begin{proof}
It is sufficient to consider the case with  $U_1$  uniformly distributed  on $\{1,2,\dots,k\}$, $V_2+U_2=k+1$ and $V_3+U_3=k$, where $k$ is some constant.
The general case will follow by shifting the range of the variables, induction and taking mixtures.


\noindent
Then   for $1\leq m\leq k$ we have
\begin{align}
\prob(V_2\geq m)  =
 \prob (V_1\geq m|V_1+U_1=k+1)= \nonumber \\
\frac{\sum_{j=m}^k \prob(V_1=j, U_1=k+1-j)}{\sum_{j=1}^k \prob(V_1=j, U_1=k+1-j)}  =
 \frac{k^{-1}\, \prob(m\leq V\leq k)}{k^{-1}\, \prob(1\leq V\leq k)} = \nonumber \\
\frac{\, \prob(m\leq V_1\leq k-1)+\prob(V_1=k)   }{\, \prob(1\leq V_1\leq k-1)+\prob(V_1=k)} \geq
       \frac{\, \prob(m\leq V_1\leq k-1)  }{\, \prob(1\leq V_1\leq k-1)}
\geq  \nonumber  \\
 \frac{k^{-1}\, \prob(m\leq V_1\leq k-1)  }
{k^{-1}\, \prob(0\leq V_1\leq k-1)}=
 \prob (V_1\geq m|V_1+U_1=k)=\prob(V_3\geq m), \nonumber
\end{align}
where we used that $(a+x)/(b+x)$ increases in $x\geq 0$ for $b\geq a>0$.  The   cases $m>k$ or $m<1$ are  trivial, and  the relation
follows.
\end{proof}

\begin{proof}[Proof of Lemma \ref{Lemma_Mallows_mono}]
We apply Lemma \ref{Lemma_stochastic_nonsense} repeatedly to the Lehmer code of permutations ${
\Pi}^*_n$, $\Pi$ and $\Pi'$, respectively. Here, $U_i$ is the last coordinate of the code and $V_i$
is the sum of all other coordinates.
\end{proof}

\begin{proof}[Proof of Theorem \ref{Theorem_Mallows}]
The argument goes along the line of proof of Theorem \ref{Theorem_Hamming} hence we omit some details.
By the strong law of large numbers for sums, under the Mallows distribution \eqref{Mallows} the
number of inversions    satisfies the almost sure asymptotics:
\begin{enumerate}
\item[(a)] $I(\Pi_n)\sim \frac{q}{1-q}\, n$ for $0\leq q<1$,
\item[(b)] $I(\Pi_n)\sim \frac{1}{4}n^2 $ for $q=1$ (the uniform case),
\item[(c)] ${n\choose 2}-I(\Pi_n)\sim \frac{q^{-1}}{1-q^{-1}}\,n$ for $1<q\leq \infty$.
\end{enumerate}

Let $\varkappa(\nu)$ be a sequence inducing an extreme ISS virtual permutation $\boldsymbol {\Pi}=(\Pi_n, n=1,2,\dots)$, as in
Theorem \ref{Theorem_Martin_approximation} (so $r(\nu)=\varkappa(\nu)$). Passing to a subsequence of the values of $\nu$ we can achieve that
either $\varkappa(\nu)\sim \frac{q}{1-q} \nu$ or ${\nu\choose 2}- \varkappa(\nu)\sim
\frac{q^{-1}}{1-q^{-1}}\nu$ for some $q$, or both $\varkappa(\nu)$ and  ${\nu\choose 2}- \varkappa(\nu)$ grow faster
than linearly as $\nu\to\infty$.

Consider the case $\varkappa(\nu)\sim \frac{q}{1-q} \nu$ with $0<q<1$. Fix $n$ and
$0<\eps<\min(q,1-q)$. To construct stochastic upper and lower bounds for $I(\Pi_n)$ we use  virtual permutations Mallows$(q\pm \epsilon)$ and appeal to Lemma
 \ref{Lemma_Mallows_mono}. Sending $\epsilon\to 0$ we conclude that $\Pi_n$ is Mallows$(q)$ for every $n$, hence
$\boldsymbol{\Pi}$ is Mallows$(q)$.

Two other  asymptotic regimes for $\varkappa(\nu)$ are treated similarly.
It follows that the family of Mallows$(q)$
virtual permutations contains all extreme ISS virtual permutations. Since by (a), (b) and
(c) they all have didjoint supports all these are extreme.
\end{proof}

It is of interest to recast Theorem \ref{Theorem_Mallows} in terms of ratios of combinatorial numbers.
The number of permutations $\sigma\in S_n$ with $k$ inversions is the
{\it Mahonian number} $M_{n,k}$ counting solutions to the equation
 $\eta_1+\dots+\eta_n=k$, where  $\eta_1,\eta_2,\dots$ are integer variables satisfying $0\leq \eta_j<j$.
The number of extensions of any such $\sigma$ to a permutation $\pi\in S_\nu$ with $I(\pi)=\varkappa$ is
a generalised Mahonian number $M_{n,k}^{\nu,\varkappa}$ counting solutions to $\eta_{n+1}+\dots+\eta_\nu=\varkappa-k$.
Identifying
\begin{eqnarray}\label{kernelMallows}
\frac{M_{n,k}^{\nu,\varkappa}}{M_{\nu,\varkappa}}
\end{eqnarray}
with the Martin kernel we obtain

\begin{Corollary}\label{CorMallows} If as $\nu\to\infty$ and $\varkappa=\varkappa(\nu)$ varies in some way
the ratios {\rm (\ref{kernelMallows})} converge for all $n$ and $0   \leq k\leq{n\choose 2}$
then the limit is
$q^k/[n]_q!$ for some $q\in[0,\infty]$.
The convergence holds if and only if
either $\varkappa(\nu)\sim \frac{q}{1-q}\, \nu$ for $0\leq q<1$, or ${\nu\choose 2}-\varkappa(\nu) \sim \frac{q^{-1}}{1-q^{-1}}\,\nu$ for $1<q\leq \infty$,
or both $\varkappa(\nu)$ and ${\nu\choose 2}-\varkappa(\nu)$ grow faster than linearly for $q=1$.
\end{Corollary}

\section{Cayley spherical symmetry}

The Cayley distance between $\pi$ and $\sigma$ in $S_n$ is  defined as  the minimal  number of
transpositions needed to transform one permutation in another.
 Right-multiplying $\pi\in S_n$ by the transposition $(i,j)$
amounts to swapping letters $i$ and $j$ in the one-row notation of $\pi$. The
multiplication  increases the number of cycles by one if $i$ and $j$ belong to the
same cycle of $\pi$, and decreases by one otherwise. Thus the distance to the
identity is $|\pi|=n-C(\pi)$, where $C(\pi)$ denotes the number of cycles in $\pi$.
We take the same projections $f_n$ as in the setting with Hamming distance of
Section \ref{Section_Hamming}. If $\sigma\in S_{n-1}$ is a permutation with $k$
cycles, then it has $n-1$ extensions with the same number of cycles and $1$
extension with $k+1$ cycles. Hence, the Cayley spherical symmetry is preserved under
these projections.

For virtual permutation $\boldsymbol{\pi}=(\pi_1,\pi_2,\dots)$, we have $C(\pi_n)=\sum_{j=1}^n \beta_j$ where
$\beta_j=1$ if $j$ is a fixed-point of $\pi_j$ and $\beta_j=0$ otherwise.
For the uniform $\boldsymbol{\Pi}^*$,
the sequence of $\beta_j$'s is independent Bernoulli with $\prob(\beta_j=1)=1/j$.
Exponential tilting with parameter $\theta\in[0,\infty]$ yields a family of ISS virtual permutations
with the {\it Ewens distribution}
\begin{equation}\label{Ewens}
\prob(\Pi_n=\pi)=\frac{\theta^{C(\pi)}}{(\theta)_n}, ~~\pi\in S_n,
\end{equation}
where $(\theta)_n=\prod_{i=0}^{n-1}(\theta+i)$. For $\theta=0$ this is a uniform
cyclic permutation, and for $\theta=\infty$ the unit mass at the identity. Under the
Ewens distribution $\boldsymbol{\Pi}$ follows the dynamics of the Chinese restaurant
process, where element $n$ is a new cycle appended to $\Pi_{n-1}$  with probability
$\theta/(\theta+n-1)$, and is inserted    next to the right in the cycle of any
$j\in[n-1]$ with probability $1/(\theta+n-1)$ \cite{CSP}.


\begin{Theorem} \label{Theorem Ewens}
The Ewens distributions {\rm (\ref{Ewens})} with $\theta\in[0,\infty]$ and only they are the
extreme  ISS virtual  permutations with respect to the Cayley distance.
\end{Theorem}

A proof of Theorem \ref{Theorem Ewens} can be found in \cite{Gibbs}, where it appears in a minor disguise
(see also \cite[Theorem
4.1]{Frick}). Following our unified approach, the key observation is the following lemma:

\begin{Lemma} \label{Lemma_Ewens_mono}
 Let $\Pi$ and $\Pi'$ be two random Cayley--spherically symmetric permutations in $S_\nu$, such that
$C(\Pi)\geq_{\rm st} C(\Pi')$. Then for every $n<\nu$,  also $C(f_{\nu\downarrow n}(\Pi))\geq_{\rm st} C(f_{\nu\downarrow n}(\Pi'))$.
\end{Lemma}
\begin{proof}
 This immediately follows from the fact that if $\pi\in S_n$ has $k$ cycles, then $f_n(\pi)$ has either
 $k$ or $k-1$ cycles.
\end{proof}

Using Lemma \ref{Lemma_Ewens_mono} and  that under the Ewens distribution with
$\theta\in (0,\infty)$ the number of cycles grows as $C(\Pi_n)\sim \theta\log n$,
see e.g.\ \cite{ABT}, the proof of Theorem \ref{Theorem Ewens} can be produced along
the same lines as in Theorems \ref{Theorem_Hamming}   and \ref{Theorem_Mallows}.
 A counterpart of Corollary \ref{CorMallows} concerns limiting regimes for the ratios of generalised Stirling numbers of the first kind.

\end{document}